\documentclass[12pt]{article}

\usepackage{amsmath,amsfonts,amssymb}
\usepackage{amsthm}
\usepackage{mathtools}
\usepackage{verbatim}
\usepackage{multicol}
\usepackage{ bbold }
\usepackage{enumitem}
\usepackage{graphicx}
\usepackage[margin=1in]{geometry}
\usepackage{fancyhdr}
\usepackage[capitalize,noabbrev]{cleveref}

\makeatletter
\def\thmhead@plain#1#2#3{%
  \thmname{#1}\thmnumber{\@ifnotempty{#1}{ }\@upn{#2}}%
  \thmnote{ {\the\thm@notefont#3}}}
\let\thmhead\thmhead@plain
\makeatother

\makeatletter
\newcommand{\tpitchfork}{%
  \vbox{
    \baselineskip\z@skip
    \lineskip-.52ex
    \lineskiplimit\maxdimen
    \m@th
    \ialign{##\crcr\hidewidth\smash{$-$}\hidewidth\crcr$\pitchfork$\crcr}
  }%
}
\makeatother

\newtheorem{theorem}{Theorem}

\newtheorem{lemma}{Lemma}

\newtheorem{definition}{Definition}
\newtheorem{question}{Question}

\newtheorem*{lemmarm}{Lemma}

\title{A Generalized Isoperimetric Inequality via\\ Thick Embeddings of Graphs}
\date{}
\author{Elia Portnoy}

\begin{document}
\maketitle

\begin{abstract}
We prove a generalized isoperimetric inequality for a domain diffeomorphic to a sphere that replaces filling volume with $k$-dilation. Suppose $U$ is an open set in $\mathbb{R}^n$ diffeomorphic to a Euclidean $n$-ball. We show that in dimensions at least 4 there is a map from a standard Euclidean ball of radius about $vol(\partial U)^{1/(n-1)}$ to $U$, with degree 1 on the boundary, and $(n-1)$-dilation bounded by some constant only depending on $n$. We also give an example in dimension 3 of an open set where no such map with small $(n-1)$-dilation can be found.
\par 
The generalized isoperimetric inequality is reduced to a theorem about thick embeddings of graphs which is proved using the Kolmogorov-Barzdin theorem and the max-flow min-cut theorem. The proof of the counterexample in dimension 3 relies on the coarea inequality and a short winding number computation.
\end{abstract}

\section{Introduction}

In this paper we will prove a generalized isoperimetric inequality for a domain diffeomorphic to a sphere. For reference, we recall the classic isoperimetric inequality.

\begin{theorem}\label{iso} If $U$ is a smooth open set in $\mathbb{R}^n$ then,

$$ vol(U) \lesssim_n vol(\partial U)^{n/(n-1)}$$

where $ x \lesssim_n y$ means there is some constant $C$, only depending on $n$, so that $x \le Cy$.
\end{theorem}

Applying Banyaga's generalization of Moser's Theorem to manifolds with boundary \cite{B} we can deduce a slightly stronger theorem when $U$ is a ball.

\begin{theorem} \label{mosthm} Let $U \subset \mathbb{R}^n$ be a smooth bounded open set that is diffeomorphic to a ball in $\mathbb{R}^n$. Let $B^n$ be a ball in $\mathbb{R}^n$ with the same volume as $U$. Then there exists a smooth map, $$F: (B^n, \partial B^n) \to (U, \partial U)$$
which has degree 1 on the boundary and $|det(DF)| = 1$.
\end{theorem}

In other words, \cref{mosthm} says we can find a map from $B^n$ to $U$ that does not expand $n$-dimensional volume. We can also try to control how much $F$ stretches the volume of $k$-dimensional subsets of the domain.

\begin{definition} Suppose $F: M \to N$ is a Lipschitz map between Riemannian manifolds that is differentiable on a set of full measure $\Omega$. We define the \textbf{k-dilation} of $F$ as

$$ Dil_k(F) = \sup_{p \in \Omega} |\Lambda^k DF(p)|$$

Where for a symmetric matrix $A$, $|\Lambda^k A|$ denotes the product of its $k$ largest eigenvalues. The geometric interpretation of this definition is as follows. If $\Sigma$ is a $k$-dimensional submanifold of $M$ then,

$$ vol(F(\Sigma)) \le  Dil_k(F)  vol(\Sigma)$$

\noindent where $vol$ stands for the $k$-dimensional volume.
\end{definition}

The first theorem in this paper improves \cref{mosthm} as follows,

\begin{theorem} \label{mthm} Let $U$ be a smooth bounded open set in $\mathbb{R}^n$ that admits a diffeomorphism to a ball in $\mathbb{R}^n$ and let $B^n_R$ stand for a standard ball of radius $R$ in $\mathbb{R}^n$. For $R = vol(\partial U)^{1/(n-1)}$ and $n\ge 4$, there exists a Lipschitz map,
$$F: (B^n_R, \partial B^n_R) \to (U, \partial U)$$
which has degree 1 on the boundary and with $ Dil_{n-1}(F) \lesssim_n 1$.
\end{theorem}

For an example of this theorem we can consider the case when $U$ is the open ellipse in $\mathbb{R}^n$ with axes lengths $(L, 1,1 \ldots 1)$, where $L \ge 1$. Then for $R = L^{1/(n-1)}$, we can take $F: B_R \to U$ to be the diffeomorphism given be a linear map. The eigenvalues of $DF$ are $(L^{(n-2)/(n-1)}, L^{-1/(n-1)} \ldots  L^{-1/(n-1)})$. Multiplying the $(n-1)$ largest eigenvalues of $DF$ we get $Dil_{n-1}(F) = 1$.
\\ \par Now we give an outline of the proof of \cref{mthm}. The first step is to construct an expanding embedding of a neighborhood of a 1-dimensional subset of $U$ into $B_R$, which will act as a skeleton for $F^{-1}$. We  give some definitions that will help us work with such embeddings.

\begin{definition} 
For some number $\omega > 0$, and some open set $U \subset \mathbb{R}^n$, define $G(U, \omega)$, to be the graph defined by intersecting $U$ with a 1-dimensional grid with side-length $\omega$. The grid's vertices are $\omega\mathbb{Z}^n$ and its edges are all the straight segments between vertices of distance of $\omega$ from each other. We can assume this grid transversely intersects $\partial U$, and we let the points where the grid intersects $\partial U$ also be verticies of $G(U, \omega)$. Sometimes we will treat $G(U, \omega)$ as a 1-dimensional subset of $U$ and sometimes we will treat it as a graph.
\end{definition}

For an example, take $U = [-3/2, 3/2]^n$. Then $G(U, 1)$ is the $n$-hypercube graph along with $n2^n$ edges between the $n$-hypercube and $\partial U$.

\begin{definition} Fix some dimension $n>0$. Call a map between two graphs, $\psi: \Gamma_1 \to \Gamma_2$, \textbf{thick}, if $\psi$ maps vertices to vertices, edges to non-trivial paths, and the pre-image of each edge in $\Gamma_2$ intersects $\lesssim_n 1$ edges of $\Gamma_1$.
\par
For $U$ an open set in $\mathbb{R}^n$, we say that a map, $\psi: \Gamma_1 \to U$, is \textbf{$\omega$-thick} if $\psi$ has image in $G(U, \omega) \subset U$ and $\psi$ is a thick map from $\Gamma_1$ to $G(U, \omega)$. 
\end{definition}
The following lemma says that we can construct a map that functions as a skeleton for $F^{-1}$.

\begin{lemma} \label{mlemma} Let $n\ge 3$ and suppose we are given a smooth bounded open set $U \subset \mathbb{R}^n$. For some small $\omega>0$, that depends on $U$, let $\Gamma = G(U, \omega)$ and $\partial \Gamma = \Gamma \cap \partial U$. Then for $R = vol(\partial U)^{1/(n-1)}$, there is an $\omega$-thick embedding, $$\psi: (\Gamma, \partial \Gamma) \to (B_R, \partial B_R)$$ 
\end{lemma}

The technique of considering a skeleton for an inverse map by using a thick embedding was previously used by Guth in \cite{GC} to construct maps with small $k$-dilation between $S^m$ and $S^n$ for $m > n$ and $k > (m+1)/2$. Since maps between spaces of different dimensions were considered, Guth was able to construct the thick embedding using a quantitative h-principle. We will rely on a different method for maps between spaces of the same dimension. First we cut our graph $\Gamma$ into pieces that roughly look like balls. We will then rearrange these pieces into a larger ball and reconnect the edges between the pieces using two lemmas. The first lemma was proven in dimension three by Kolmogorov and Barzdin and generalized to higher dimensions by Guth.  

\begin{lemma} \cite{KB, GQ} \label{kblemma} Let $\Gamma$ be a graph which is a disjoint union of edges and let $V\Gamma$ stand for the vertices of $\Gamma$. Suppose that for $R = |V\Gamma|^{1/(n-1)}$ we are given an embedding,
$$\psi: V\Gamma \to \partial B^n_R$$
Then we can construct a 1-thick map,
$$\Psi: \Gamma \to B^n_R$$
which extends $\psi$.
\end{lemma}

This lemma roughly says that if want to map a graph of bounded degree, $\Gamma$, into $B_R$, without too many edges intersecting any unit ball, then we are free to chose the positions of verticies on $\partial B_R$, as long as they are sufficiently spaced out. Kolmogorov and Barzdin proved this lemma using a probabilistic method, and their bounds are sharp for expander graphs. For the proof of $\cref{mlemma}$ we will sometimes be forced to map vertices into the interior of $B_R$, which presents an obstacle for their methods. The second lemma gives us more flexibility in mapping half of the vertices of $\Gamma$  at the cost of having less flexibility over where the other half of the vertices gets mapped to.

\begin{lemma} \label{elemma} Let $\Gamma$ be the bipartite graph on two sets of vertices $S$ and $T$ with each vertex having degree 1. That is, $\Gamma$ is a set of $|S| = |T|$ disjoint edges. Suppose we are given a 1-thick map $\psi: S \to B^n_R$ so that,

$$|\psi(S) \cap B_r| \lesssim_n r^{n-1}$$ 

for each $B_r \subset B_R$ with $r>1/2$. Then we can extend $\psi$ to a 1-thick map, $$\Psi: (\Gamma, T) \to (B_R, \partial B_R)$$
\end{lemma}

The proof of \cref{elemma} combines the min-cut max-flow theorem with a good ball type argument used in a proof of the isoperimetric inequality (see \cite{W}). We use \cref{kblemma} to reconnect the edges between the balls, which we cut our graph into, and we will use \cref{elemma} to ensure that the edges adjacent to $\partial \Gamma$ get mapped to edges adjacent to $\partial B_R$. This completes the outline of \cref{mlemma}.

\par Let's now sketch how \cref{mlemma} implies \cref{mthm}. Let $W \subset B_R^n$ be a $\omega/10$ neighborhood of $\psi(G(U, \omega))$ where $\psi$ is as in \cref{mlemma}. We can construct a $C_1 \lesssim_n 1$ bilipschitz embedding from $W$ to $U$. Because we are in dimensions $\ge 4$, any two maps from a graph into $B_R^n$ are isotopic. This means we can extend $\psi^{-1}$ to a Lipschitz map $F_1: B^n_R \to U$. Let $\Phi: B^n_R\to B^n_R$ be a $C_2 \lesssim_n 1$ Lipschitz map that contracts the complement of $F_1(W)$ to an $(n-2)$-dimensional set and let $F = \Phi \circ F_1$. Observe that $F(B^n_R - W)$ is a subset of an $(n-2)$-dimensional set and that $F$ is $C_1C_2$ Lipschitz on $W$. Putting these observations together we see that $Dil_{n-1}(F) \lesssim_n 1$.
\\ \par
In $\cite{GC}$, Guth showed that if $k$ is small compared to the dimension of the domain, there are some topological restrictions to the existence of maps with small $k$-dilation. This suggests that \cref{mthm} might fail in lower dimensions. We show that there are indeed 3-dimensional sets for which \cref{mthm} fails to hold.

\begin{theorem} \label{wthm} For any $N \in \mathbb{N}$, there is an open set $U_N \subset \mathbb{R}^3$, diffeomorphic to a ball and with $ vol_2(\partial U_N) \lesssim 1$, so that any $F: (B^3_1, \partial B^3_1) \to (U_N, \partial U_N)$, with degree 1 on the boundary, must have $ Dil_2(F) \gtrsim N^{1/2}$. 
\end{theorem} 

Our $U_N$ will be an open set whose boundary winds around some axis $N$ times. Using the coarea formula we will show that any map that does not increase area from $B^3$ to $U_N$ must map a disc with small area to a disc which winds $N$ times around some axis. Then a winding number computation will imply that the area of this image of this disc is $\gtrsim N^{1/2}$.\\ \par
We end this section with some open questions,

\begin{question} The isoperimetric inequality holds for cycles of any dimension in $\mathbb{R}^n$. Is there an analogue of \cref{mthm} in higher codimensions?  More precisely, Let $W \subset \mathbb{R}^n$ be the image of an embedding (or immersion) of $\partial B^{m}$ into $\mathbb{R}^n$. For $R = vol_{m-1}(W)^{1/(m-1)}$ can we construct a map, 

$$F: (B^m_R,  \partial B^m_R) \to (\mathbb{R}^n, W)$$

with degree 1 on the boundary and $Dil_{m-1}(F) \lesssim_n 1$? 
\end{question} 

\begin{question} Are there similar isoperimetric inequalities for other $k$-dilation? For instance, suppose $U$ is as in \cref{mthm} and we are given a degree 1 map $f: \partial B^n_R \to \partial U$ with $ Dil_k(f) \le 1$ for some $k>0$. For which $k'$ can we construct a $F: (B^n_R, \partial B^n_R) \to (U, \partial U)$, with degree 1 on the boundary, and with $ Dil_{k'}(F) \lesssim_n 1$?
\end{question}

\begin{question} Suppose we are given a graph of bounded degree, $\Gamma$, and let  $Spec(\Gamma)$ be the spectrum of its laplacian. We wonder what is the smallest $R(Spec(\Gamma))>0$ so that there is there a 1-thick embedding, 
$$\psi: \Gamma \to B^3_{R(Spec(\Gamma))}$$
Are there some $\Gamma_1, \Gamma_2$ with similar spectra, but $R(Spec(\Gamma_1)$ is much bigger than $R(Spec(\Gamma_2)$? 
 \end{question}
 
\par 
This paper is organized as follows. In \cref{s1} we give a proof of \cref{mlemma} and show how it implies \cref{mthm}. In \cref{s2} we prove \cref{elemma} using min-cut max-flow. In \cref{s4} we present the example in dimension 3 from \cref{wthm}.
\section*{Acknowledgements}

The author would like to thank Larry Guth for suggesting the isoperimetric inequality involving $k$-dilation and several other valuable discussions that led to this paper.

\section{Proof of Theorem 1} \label{s1}

\begin{proof}[Proof of \cref{mlemma}]
Our strategy is to partition $U$ by cutting it into well-behaved subsets. Then we will isometrically embed these subsets into disjoint balls near $\partial B_R$. In partitioning $U$ we had to cut some edges. Using the Kolmogorov-Barzdin theorem we can reconnect these edges inside $B_R$. Finally, we will add edges connecting the image of $\partial \Gamma$ to $\partial B_R$ using \cref{elemma} from \cref{s2}. 
\\ \par
There is a technical issue we have to address before moving on with the outline above. We will construct our partition of $U$ inductively by cutting away pieces of $U$. The worry is that after cutting many pieces off from $U$ we may be left with a set that is much less regular than the one we started with. In order to avoid this we will take our set $U$ and all of the pieces in the partition to be unions of $n$-cells from a lattice. 
\par Let $G_{\omega}$ be a uniform $n$-lattice of side length $\omega$ in $\mathbb{R}^n$. For small enough $\omega>0$ there is a $C\lesssim_n 1$ bilipschitz map from $\overline U$ to a union of closed $n$-cells from $G_{\omega}$, with degree 1 on the boundary. Since our problem is scale invariant and invariant under $C\lesssim_n 1$ bilipschitz maps we can take $\omega=1$ and assume that $\overline U$ is a union of closed $n$-cells from $G_1$, without loss of generality. 

\begin{definition} A $G$-ball, denoted $B_r^G(p)$, is the union of closed $n$-cells from $G_1$ that have non-trivial intersection with $\overline{B_r(p)}$. For ease of notation we will sometimes omit the explicit dependence on the center of the ball and simply write $B_r^G$.
\end{definition}

We introduce the following notion of a good $G$-ball,
\begin{definition} Later in the proof we will fix constants $A > 1, \tilde A>1$ and $0 < \delta = \delta(A) < 1$ that only depend on $n$. Call $B^G_r$ a good $G$-ball for an open set $W \subset \mathbb{R}^n$ if the following conditions hold:
 \begin{equation} \label{eq:1}
vol(\partial W \cap B^G_r) > A\,vol(W \cap \partial B^G_r)
\end{equation}

\begin{equation} \label{eq:2} 
 vol(\partial W \cap B^G_r) > \delta r^{n-1}
 \end{equation}
 
\begin{equation} \label{eq:3}
\text{For $s>1/2$ and every ball  } B^G_s \subset B^G_r \text{,  }
 vol(\partial W \cap B^G_s) < \tilde A  A s^{n-1}
 \end{equation}
 \end{definition}

Condition \ref{eq:1} tells us that if we remove $W \cap B^G_r$ from $W$, then we decrease its boundary. This will be important in inductively constructing our partition. Condition \ref{eq:2} says that the volume of the $\partial B^G_r$ is still comparable to the volume of $\partial W \cap B^G_r$. This will help us fit all the balls we used to partition $W$ into disjoint balls near $\partial B_R$. Condition \ref{eq:3} says that $\partial W$ is not too concentrated anywhere inside our ball, which will become important when we want to use \cref{elemma} later on. \\ \par

The main feature of our partition of $U$ will be that each piece lies inside a good $G$-ball and the sum of the boundaries of these $G$-balls is not much bigger than $vol(\partial U)$. We show that such a partition exists in the following lemma.

\begin{lemma} \label{plemma} Let $U$ be a union of $n$-cells from $G_1$. There exists a partition of $U$, denoted by $\mathcal{P} = \{U_i\}$, and a set of good $G$-balls for $U$, denoted by $\{B^G_{r_i}\}$, so that,
$$U_i \subset B^G_{r_i}$$
$$|\partial \mathcal{P}| := \sum r_i^{n-1} < A'\, vol(\partial U)$$

\noindent where $A'>1$ is some large constant that only depends on $A$ and $\delta$. 
\end{lemma}

\begin{proof}[Proof of \cref{plemma}]
The proof is by induction on $vol(U)$. For the base case we consider $vol(U) = 1$ for which we can take $\mathcal{P} = \{U\}$. Note that $U$ lies in a good $G$-ball $B^G_2(p)$ for $p$ in the center of $U$. Condition \ref{eq:1} is satisfied because $vol(U \cap \partial B^G_r)=0$. Condition \ref{eq:2} is satisfied for small enough $\delta$, and condition \ref{eq:3} is satisfied for large enough $A$.

\par Now lets assume the inductive hypothesis holds for open sets with volume less than $vol(U)$ and show that we can find a $G$-good ball for $U$. To do this we will consider three cases.
\\ \\
\textbf{Case 1:} Suppose first that there is a ball $B^G_r$ with $vol(\partial U \cap B^G_r) > \tilde A A r^{n-1}$. Let $B^G_t$ be a minimal $G$-ball satisfying $vol(\partial U \cap B^G_t) > \tilde A A t^{n-1}$. In other words, there is no $B^G_s$ strictly inside $B^G_t$ that also satisfies the same inequality.  For $A$ large enough, we have $t > n$, and so $B^G_{t-n} \subsetneq B^G_{t}$ satisfies condition \ref{eq:3}. Set $B' = B^G_{t-n}$ and observe that for $\tilde A$ large enough,

$$vol(\partial U \cap B') > \tilde AA t^{n-1} - \tilde At^{n-1}>Avol(\partial B')$$

\noindent and so $B'$ satisfies conditions \ref{eq:1} and \ref{eq:2}. Thus, $B'$ is a good $G$-ball for $U$. For all the following cases we assume that $U$ satisfies condition \ref{eq:3}. 
\\ \\
\textbf{Case 2:} Suppose there exists a $p \in U$ and an $0 < r < \delta^{-1/(n-1)}$ so that $vol(\partial B^G_r(p) \cap U) = 0$. Condition \ref{eq:1} and \ref{eq:2} are easily satisfied, so $B^G_r(p)$ is a good $G$-ball for $U$. 
\\ \\
\textbf{Case 3:} Thus, we can take any $p \in U$ and assume that $vol(\partial B^G_r(p) \cap U) > 0$ for all $0 < r < \delta^{-1/(n-1)}$. Define,

$$t = \inf \{r \ge \delta^{-1/(n-1)}: vol(U \cap \partial B^G_r(p)) < \delta r^{n-1}\}$$

\noindent Let $B' = B^G_{t}(p)$. Suppose for contradiction that $vol(B' \cap \partial U) <  A\delta t^{n-1}$. Then by the isoperimetric inequality we have,
$$vol(B' \cap U) \lesssim_n (vol(B' \cap \partial U)+vol(\partial B' \cap U))^{n/(n-1)} \lesssim_n (A\delta)^{n/(n-1)} t^n$$
However, by the minimality property of $t$ we have,

$$vol(B' \cap U) \ge \sum_{k = 1}^{t/3-1} vol(\partial B^G_{3k}(p) \cap U) \gtrsim_n \delta t^n$$

\noindent Picking $\delta > 0$ small enough so that  $(A\delta)^{n/(n-1)} \lesssim_n \delta$ we get a contradiction. Thus, $vol(B' \cap \partial U) \ge  A\delta t^{n-1}$ and so $B'$ is a good ball.
\\ \par

We have now shown that we can always find a good $G$-ball $B' = B^G_t(p)$ for $U$. If $U \subset B'$ we complete the inductive step by taking $\mathcal{P} = \{B'\}$, otherwise let $U_2 = U - B'$. Since\\ $vol(U_2) < vol(U)$ we can apply the inductive hypothesis to $U_2$ to get a partition $\mathcal{P}_2$ of $U_2$, so that each element of the partition is inside a good $G$-ball, with,

$$|\partial \mathcal{P}_2|  < A' vol(\partial U_2) \le A'(vol(\partial U) - vol(\partial U \cap B') + vol(\partial B' \cap U)) $$
$$ < A'vol(\partial U) - A'(1-1/A)vol(\partial U \cap B') < A'vol(\partial U) - A'(1-1/A)\delta t^{n-1}$$

\noindent The first inequality comes from the inductive hypothesis, the third comes from Condition \ref{eq:1}, and the fourth from Condition \ref{eq:2}. Then we can take our partition of $U$ to be $\mathcal{P} = \mathcal{P}_2 \cup \{B' \cap U\}$ and compute,
$$|\partial \mathcal{P}| = |\partial \mathcal{P}_2| + t^{n-1} < A'vol(\partial U) - (A'(1-1/A)\delta-1) t^{n-1}< A'vol(\partial U)$$
where the last inequality holds for $A' = A'(A, \delta)$ large enough. This completes the inductive step and the proof of \cref{plemma}.
\end{proof}

We now begin describing the 1-thick map $\psi$ from the statement of \cref{mlemma}. This will be done in a few steps and Figure 1 below may be helpful for visualizing the steps. Let $\mathcal{P} = \{U_i\}$ and $\{B^G_{r_i}\}$ be the partition and good $G$-balls for $U$, given by \cref{plemma}. According to \cref{plemma},

$$\sum_i r_i^{n-1} \le A' vol(\partial U)$$

 So for some $C \lesssim_n 1$ large enough, and $R = C vol(\partial U)^{1/(n-1)}$, there exists points $\{p_i\} \subset B_R$ so that $\{B_{2r_i}(p_i)\}$ are disjoint balls inside $B_R$ with disjoint radial projections to $\partial B_R$. We also denote,
 
$$\Sigma = \bigcup_i \partial U_i$$

\noindent \textbf{Step 1:} We begin by describing our map $\psi$ on $\Gamma \cap U_i$. Let $\Gamma_i \subset \Gamma$ be the set of edges in $U_i$, that don't intersect $\partial U_i$. Since $\Gamma_i \subset B^G_{r_i}$ we can let $\psi$ map $\Gamma_i$ into $B_{2r_i}(p_i)$ isometrically, with some translation. 
\\ \\
\noindent \textbf{Step 2:}  Let $E_i$ be the set of edges in $\Gamma$ that intersect $\partial U_i$.  Each edge in $E_i$ will be mapped to a piece-wise linear path in $B_R$ composed of several parts. To do this lets recall the statement of \cref{elemma},

\begin{lemmarm} Let $\Gamma'$ be the bipartite graph on two sets of vertices $S$ and $T$ with each vertex having degree 1. Suppose we are given a 1-thick map $\psi: S \to B^n_{R'}$ so that,

$$|\psi(S) \cap B_r| \lesssim_n r^{n-1}$$ 

for each $B_r \subset B_{R'}$ with $r>1/2$. Then we can extend $\psi$ to a 1-thick map, $$\Psi: (\Gamma', T) \to (B_{R'}, \partial B_{R'})$$
\end{lemmarm}

Let $S_i$ be the endpoints of $E_i$ in $U_i$. Notice that $S_i$ lies in the 1-neighborhood of $\partial U_i$, and since $B_{r_i}^G$ is a good G-ball for $U_i$ it satisfies condition \ref{eq:3}. Thus, the assumption of \cref{elemma} is satisfied for $S = S_i$, $B_{R'} = B_{2r_i}(p_i)$, and the $\psi$ we constructed in the previous step which is defined on $S_i$. We now use \cref{elemma} to construct the first part of these paths, which connect $S_i$ to $\partial B_{2r_i}(p_i)$.

\par The next part of these paths, will consist of a line segment connecting $\partial B_{2r_i}(p_i)$ with $\partial B_R$, in the radial direction. We finish Step 2 by applying this procedure for every $i$. Notice that this procedure  gives us a 1-thick embedding.
\par
Finally, notice that since $\partial \Gamma \subset \Sigma$, we have mapped every edge adjacent to $\partial \Gamma$ to a path with an endpoint in $\partial B_R$. 
\\ \\
\noindent \textbf{Step 3:} Since we can make a $C \lesssim_n 1$ bilpschitz perturbation of $\Gamma$ we can assume that $V\Gamma \cap \Sigma = \partial \Gamma$. Let $E_\Sigma$ be the set of edges that intersect $\Sigma$ but are not adjacent to $\partial \Gamma$. In this step, we'll finish defining $\psi$ on $E_{\Sigma}$.
\par
Each edge $e \in E_{\Sigma}$, has one endpoint $s_i \in U_i$ and another endpoint $s_j \in U_j$, for some $j \neq i$. In the previous step we started defining the image of $e$ to be two paths, one path going from $s_i$ to some $t_i \in \partial B_R$ and another path going from $s_j$ to some $t_j \in \partial B_R$. In order to finish defining our map on $e$ we need to find a path from $t_i$ to $t_j$. To extend $\psi$ to $E_{\Sigma}$, lets recall the statement of \cref{kblemma},

\begin{lemmarm} Let $\Gamma'$ be a graph which is a disjoint union of edges. Suppose that for $R^{n-1} \lesssim_n |V\Gamma|$ we are given an embedding,
$$\psi: V\Gamma' \to \partial B_R$$
Then we can construct a 1-thick map,
$$\Psi: \Gamma' \to B_R$$
which extends $\psi$.
\end{lemmarm}

Observe that we have,

$$|\partial E_{\Sigma}| \lesssim_n \sum_i |E_i| \lesssim_n  \sum_i |U_i| \lesssim_n  \sum_i r_i^{n-1} \lesssim_n R^{n-1}$$

where the second to last inequality follows from condition \ref{eq:3} of a good G-ball, and the last inequality comes from \cref{plemma}. Note that $\psi$ was defined on $\partial E_{\Sigma}$ in the previous steps. Thus, applying \cref{kblemma} with $\Gamma' = E_{\Sigma}$, we can extend our $\psi$ to a 1-thick map on $E_{\Sigma}$. This completes the construction of $\psi$.

\end{proof}

\begin{figure}[h]
\centering
\includegraphics[width=0.65\textwidth]{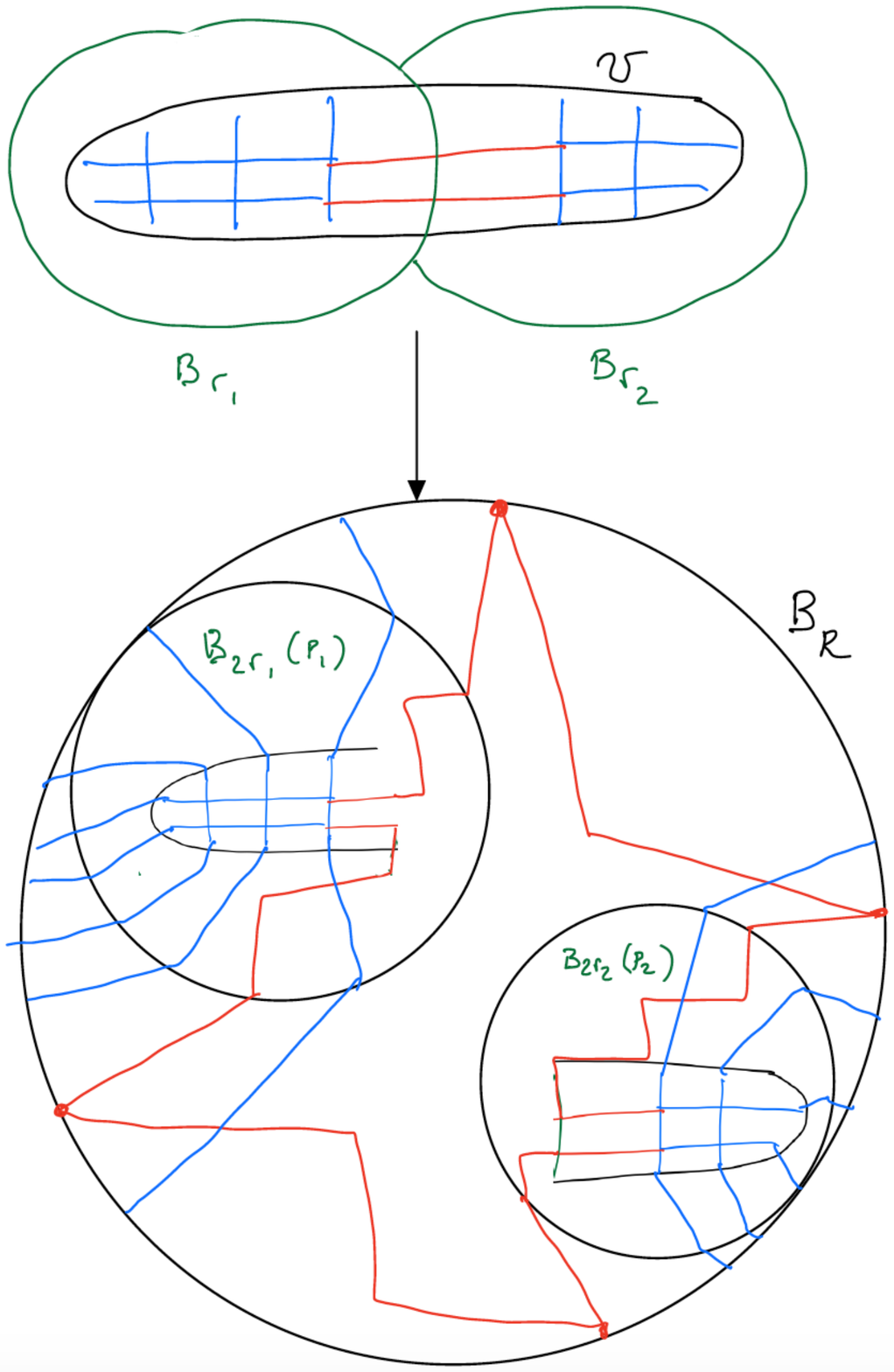}
\caption{Sketch of $\psi$ for the case of a partition consisting of 2 sets. G-good balls from \cref{plemma} are shown in green. $E_{\Sigma}$ is shown in red. The rest of the edges of $\Gamma$ are shown in blue.}
\end{figure}

We now give a technical lemma which allows us to modify a thick maps of graphs to an expanding embeddings of its neighborhood.

\begin{lemma} \label{tlemma} Let $U, U' \subset \mathbb{R}^n$ be open sets, which are unions of $n$-cells from $G_1$, and suppose $\psi: G(U, 1) \to U'$ is a $1$-thick map. For a graph $G$ embedded a subset in $\mathbb{R}^n$, denote its $r$-neighborhood as $N_r(G)$. Then for some $\delta \gtrsim_n 1$ and $r = 1/10$, there is a $\delta$-expanding embedding,
$$\Psi: N_{r}(G(U, 1)) \to U'$$
In other words, $|D\Psi(v)| \ge \delta |v|$ for every tangent vector $v$. This map satisfies, 
$$dist(\Psi(p), \psi(p)) \lesssim_n 1$$ for every $p \in G(U, 1)$.
\end{lemma}

\begin{proof} Denote by $K_m$ the complete graph on $m$ verticies. For some $m \lesssim_n 1$, let $T$ be the graph where we replace each vertex of $G(U', 1)$ by $K_m$, and where we replace each edge in $G(U',1)$ by $m$ disjoint paths between the verticies of two complete graphs, so that the degree of any vertex of $T$ is $\lesssim_n 1$. We take $m \lesssim_n 1$ large enough, so that we can lift our thick map $\psi$ to a genuine embedding,
$$\psi': G(U, 1) \to T$$

Let $m'U \subset \mathbb{R}^n$ denote the set $U$ scaled by $m' > 0$. Then for $m' \lesssim_n 1$ large enough there is an embedding,

$$\phi: T \to m'N_{r}(G(U', 1))$$

\noindent so that for any vertices $v, v' \in VT$, 
$$dist(\phi(v), \phi(v') \ge 1$$
and for any edges $e, e' \in ET$ that do not share any endpoints,
$$dist(\phi(e) , \phi(e')) \ge 1$$

We can then construct a 1-expanding embedding $\Psi: N_{r}(G(U, 1)) \to m'N_{r}(G(U', 1))$ so that $\Psi|_{G(U, 1)} = \phi \circ \psi'$. Shrinking the image of $\Psi$ by a factor of $1/m'$, we get a $\delta = 1/m'$ expanding embedding as desired.
\end{proof}

We continue to the proof of \cref{mthm}.

\begin{proof}[Proof of \cref{mthm}] Let $U \subset \mathbb{R}^n$ be our open set, which we can assume to be a union of $n$-cells from $G_1$ and let $R = C vol(\partial U)^{1/(n-1)}$ for some large $C \lesssim_n 1$. Let $W' = N_{1/10}(G(U, 1))$. Using \cref{mlemma} and \cref{tlemma} we can construct an $\delta \gtrsim_n 1$  expanding embedding,

$$\Psi: (W', W' \cap \partial U)  \to (B_R, \partial B_R)$$

\noindent Letting $F_1 = \Psi^{-1}$, we can then extend $F_1$ to,

$$F_2: (Im(\Psi) \cup \partial B_R, \partial B_R) \to (U, \partial U)$$

\noindent so that $F_1$ is $C \lesssim_n 1$ Lipschitz and has degree 1 on the boundary. Next we extend $F_1$ to all of $B_R$ using an argument similar to Lemma 11.2 from \cite{GC}. 
\par If $n>3$ any two embeddings of a graph into $B^n_R$ are isotopic. If in addition both embeddings map some vertices to $\partial B_R$, we can ensure that those vertices are embedded in $\partial B_R$ throughout the isotopy. By our hypothesis there is some diffeomorphism, 
$$F_3: (B_R, \partial B_R) \to (U, \partial U)$$

Thus, since we assumed $n>3$ we can extend $F_2$ to a diffeomorphism  $F_4:  (B_R, \partial B_R) \to (U, \partial U)$.

\par Let $\Phi: (U, \partial U) \to (U, \partial U)$ be a $C \lesssim_n 1$ Lipschitz map with degree 1 on the boundary which shrinks $U - N_{1/10}(G(U,1))$ to a $(n-2)$-dimensional set $\Sigma$ (see Lemma 11.1 of \cite{GC}). Now let $F = \Phi \circ F_4$. On the the pre-image of $\Sigma$, the $(n-1)$-dilation of $F$ is 0 because $dim(\Sigma) < n-1$ and on the pre-image of $U - \Sigma$, $F$ is $C \lesssim_n 1$ Lipschitz by the properties of $F_4$ and $\Phi$. Thus, $ Dil_{n-1}(F) \lesssim_n 1$.

\end{proof}

\section{Applications of Min-Cut Max-Flow to Thick Embeddings. Proof of \cref{elemma}.} \label{s2}

The proof of \cref{elemma} will rely on an integral max-flow min-cut theorem which can be derived from the Ford-Fulkerson algorithm. The version used here is stated as follows,

\begin{theorem}\label{mfmclemma} \cite{E} Let $\Gamma$ be a directed graph with vertices $V\Gamma$ and edges $E\Gamma$. Let $c: E\Gamma \to \mathbb{Z}_{\ge0}$ be some given function, we call the capacity. Let $S,T \subset V\Gamma$ be two disjoint subset of vertices where each $v \in S$ has only edges directed out of it and each $v \in T$ has only edges directed into it. A flow is a function $f: E\Gamma \to \mathbb{Z}_{\ge0}$ which satisfies the following conditions,
\\ \\
Capacity Constraint: For each $e \in E\Gamma$, 
\begin{equation}f(e) \le c(e)
\end{equation}
Conservation of Flow: For each $v \in V\Gamma - (S \cup T)$,
\begin{equation} \sum_{w\in V\Gamma, [w,v] \in E\Gamma} f([w,v]) = \sum_{u\in V\Gamma, [v, u] \in E\Gamma} f([v, u])
 \end{equation}
Where $[v_1,v_2]$ is an edge directed from $v_1$ to $v_2$.
 \\ \par
A cut $\Sigma \subset E\Gamma$ for the sets $S$ and $T$ is a set of edges so that the graph,

$$\Gamma - \Sigma - \bigcup_{e \in E\Gamma, c(e)=0}e$$

 has no directed path from any vertex in $S$ to any vertex in $T$. Define the capacity of a cut $\Sigma$ to be $c(\Sigma) = \sum_{e \in \Sigma} c(e)$ and the capacity of a flow $f$ to be $\sum_{e: \partial e \cap S \neq \{\}} f(e)$. Then the maximum capacity of a flow is equal to the minimum capacity of a cut. 
\end{theorem}

We first prove a lemma that looks similar to $\cref{elemma}$.
\begin{lemma} Let $\Gamma$ be the bipartite graph on two sets of vertices $S$ and $T$ with each vertex having degree 1. That is, $\Gamma$ is a set of $|S| = |T|$ disjoint edges. Suppose we are given a 1-thick map $\psi: S \to B^n_R$ so that for any open set $U \subset B_R$,

$$|\psi(S) \cap U| \lesssim_n vol(\partial U)$$ 

Then we can extend $\psi$ to a $1/2$-thick map, $$\Psi: (\Gamma, T) \to (B_R, \partial B_R)$$
\end{lemma}

\begin{proof} We first construct an oriented graph with capacities. Let $G$ be the oriented graph that has the same vertex set as $G(B_R^n, 1/2)$ but has an edge in each direction between adjacent vertices in $G(B_R^n, 1/2)$. So $|EG| = 2|EG(B_R^n, 1/2)|$. If we treat edges in different directions as the same segment in $B_R^n$, we can still think of $G$ as a subset of $B_R^n$. Let $\partial G = \partial B_R^n \cap VG$. 
Define a capacity function $c: EG \to \mathbb{Z}_{\ge0}$ so that,

\begin{enumerate}
\item For each $v \in \psi(S)$, let $\sum_{w \in VG: [v,w] \in EG} c([v,w]) = |\psi^{-1}(s)|$.
\item For all $v \in \psi(S)$ and $[w,v] \in EG$, we have $c([w,v]) = 0$
\item For all $v \in \partial G$ and $[v,w] \in EG$, we have $c([v,w]) = 0$.
\item For all other edges in $EG$, we set the capacity be some large $M > 0$, to be defined later. 
\end{enumerate}

Let's show that a minimal cut separating $\psi(S)$ from $\partial G$ has capacity at least $|S|$. We apply an induction on $|S|$ to take care of the cases when the cut passes near $\psi(S)$. The base case $|S| = 1$ is easy to see. We continue to proving the inductive hypothesis.
\par Let $ES$ denote the set of edges adjacent to $\psi(S)$ and let $\Sigma$ be a minimal cut. If $\Sigma \cap ES \neq \emptyset$, there is an $s \in S$ and $e = [\psi(s), v] \in \Sigma$ with $c(e)>0$. Consider a modified min-cut max-flow problem where we change $S$ to $S' = S - s$ and we change $c(e)$ to $c'(e) = c(e) - 1$, but keep everything else the same. Notice that $\Sigma$ is still a cut for the two sets $\psi(S')$ and $\partial G$. Also notice that the hypothesis of the lemma are satisfied since,

$$|\psi(S') \cap U| \le |\psi(S) \cap U| $$

\noindent for every open set $U$. Now we can apply our inductive hypothesis to get $c'(\Sigma) \ge |S'| = |S|-1$. Since $e \in \Sigma$ we see that $c(\Sigma) = c'(\Sigma)+1 \ge |S|$.
\par Thus, we can assume that $\Sigma \cap ES = \emptyset$. Let $H \subset G$ be the component of $G - \Sigma$ containing $\psi(S)$. We define an open set $U$ to be the $1/2$ neighborhood of $H$ and observe that,

$$vol(\partial U) \lesssim_n |\Sigma|$$

Since $\Sigma \cap ES = \emptyset$ we have $c(\Sigma) = M|\Sigma|$. Putting this together with the assumption of the lemma we have, for some $C \lesssim_n 1$ and $M = M(C) \lesssim_n 1$, large enough,

$$|S| \le C vol(\partial U) \le C |\Sigma| = \frac{C}{M} c(\Sigma) \le c(\Sigma)$$

This completes the inductive hypothesis. By \cref{mfmclemma} there is a flow $f$ that takes values in $\mathbb{Z}_{\ge 0}$ and has capacity $|S|$. Since $f$ satisfies the conservation property at each vertex, we can turn $f$ into a set of paths that connect each vertex in $S$ to some vertex in $T$. The capacity constraint on $f$ says that at most $M$ of these paths intersect at each edge. This gives us a way to extend $\psi$ to a 1-thick embedding $\Psi$ and completes the proof of the lemma.

\end{proof}

Now we show that the hypothesis in the previous lemma can be weakened to that of $\cref{elemma}$. The proof 
uses a good ball type argument common in metric geometry.

\begin{lemma} Let $S$ be a finite set of points in $B_R^n$ so that,

$$|S \cap B_r| \lesssim_n r^{n-1}$$ 

for each ball $B_r \subset B_R$ with $r>1/2$. Then for any open set $U\subset \mathbb{R}^n$,

$$|S \cap U| \lesssim_n vol(\partial U)$$ 
\end{lemma}

\begin{proof} Suppose we are given our open set $U$. We begin by defining a cover for $U$. For some small $1/4 > \delta >0$ and each $p \in U$ let,

$$t(p) = \inf \{r > 0: vol(\partial B_r(p)\cap U) < \delta r^{n-1}\}$$ 

\noindent and denote $B(p) = B_{t(p)}(p)$. From the definition of $t(p)$ we see that, 

\begin{equation} \label{eq:4}
vol(\partial B(p) \cap U) = \delta \, vol(\partial B(p))
\end{equation}

Since $\partial B(p) \cap \partial U$ is an embedded $(n-2)$-cycle in an $S^{n-1}$ we can assume its minimal filling is a sum of connected components of $\partial B(p) - (\partial B(p) \cap \partial U)$. There are only two such possible fillings. Since $\delta < 1/4$, equation \ref{eq:4} tells us that $\partial B(p) \cap U$ is a minimal filling of $\partial B(p) \cap \partial U$ on $\partial B(p)$. 
\par We now use an observation of Federer-Fleming from \cite{F}. Their observation is that if we radially project a $(n-1)$-dimensional set $Z \subset B(p)$, from a random point in $B(p)$ to $\partial B(p)$, then with high probability, the volume of the image of $Z$ under the projection is $\lesssim_n vol(Z)$. We can use such a projection, to project $B(p) \cap \partial U$ to a filling of $\partial B(p) \cap \partial U$ on $\partial B(p)$ with volume $\lesssim_n vol(B(p) \cap \partial U)$.  Since $\partial B(p) \cap U$ was a minimal filling for $\partial B(p) \cap \partial U$ we have,

\begin{equation} \label{eq:5}
vol(\partial B(p) \cap U) \lesssim_n vol(B(p) \cap \partial U)
\end{equation}

We now use the Besicovitch Covering theorem to find a finite set of points $\{p_i\}\subset U$ so that the set of balls $\{B(p_i)\}$ have multiplicity bounded by $C \lesssim_n 1$ and cover all of $U$. Putting the above observations together we obtain,

$$vol(\partial U) \gtrsim_n \sum_i vol(\partial U \cap B(p_i)) \gtrsim_n \sum_i vol(U \cap \partial B(p_i))$$
$$= \sum_i \delta r^{n-1} \gtrsim_n \sum_i \delta |B(p_i) \cap S'| \gtrsim_n \delta |S'|$$
 
The second inequality comes from equation \ref{eq:5} and the second to last inequality comes from the assumption on $S$ given in the hypothesis of the lemma. 
\end{proof}

\begin{proof}[Proof of \cref{elemma}] The proof follows by combining the two lemmas above.
\end{proof}

\section{Proof of \cref{wthm}}  \label{s4}

Our example in dimension 3 will make use of the coarea formula. We state it here for reference.

\begin{lemma} For $m \le n$ and an open set $U \subset \mathbb{R}^n$, let $f: U \to \mathbb{R}^m$ be a Lipschitz map. Then,

$$\int_{U} |\Lambda^m DF(x)| dx \gtrsim_n \int_{\mathbb{R}^m}  vol_{n-m}(f^{-1}(y) \cap U) dy$$
\end{lemma}

\begin{proof}[Proof of \cref{wthm}]
 Given $N \in \mathbb{N}$, our goal is to construct a set $U_N \subset R^3$ of $\textup{area}(\partial U_N) \lesssim 1$ so that any Lipschitz map $$F: (B_1^3, \partial B^3) \to (U_N, \partial U_N)$$ with degree 1 on the boundary and with $D =  Dil_2(F)$, will have $D \gtrsim N^{1/2}$.
 \\ \par
 We begin by defining two open sets inside $B_1^3$. For some small $\delta > 0$, let $T_1$ be an embedded tube which is a $\delta$-neighborhood of an embedded curve that sits inside $\{1/2 \le x^2 + y^2 \le 3/4\}$ and spins $N$ times around the $z$-coordinate axis. Also ensure that $T_1$ intersects the upper hemisphere of $\partial B_1$ in one small disc. Let $T_2$ be a $\delta$-neighborhood of an embedded $B_1^2$, say defined by $\{x^2 + y^2 \le 1/2, \,z \in [1/2-2\delta, 1/2]\}$. We ensure that $T_2$ intersects $T_1$ in a small neighborhood of some point. For $T = T_1 \cup T_2$, we let $U_N = B_1^3 - T$.  Note that $\partial U_N$ is diffeomorphic to $\partial B_1^3$ and $\textup{area}(\partial U_N) \lesssim 1$ as long as $\delta>0$ is small enough. The set $U_2$ is illustrated in Figure 2 below.
 \\ \par 
We now show that $D\gtrsim N^{1/2}$. Define $\pi: B_1^3 \to B_1^2$ to be the projection in the $z$-direction and let $U' = U_N \cap \{x^2 + y^2 \le 1/3, \,z \ge 1/2\}$, in other words, the region above the middle of $T_2$. By the coarea formula applied to $\pi \circ F|_{F^{-1}(U')}$, there is a point in $B^2_1$ whose pre-image is a segment $\gamma' \subset B_1^3$ with $\textup{length}(\gamma') \lesssim D$ and endpoints on $\partial B_1^3$. Note that one of the end points of $F(\gamma')$ lies on $T_2$ and the other lies in the upper hemisphere of $\partial B^3_1$. Let $\gamma''$ be the geodesic segment in $\partial B_1^3$ between the endpoints of $\gamma'$ and define a loop $\gamma = \gamma' \cup \gamma''$. Note that $F(\gamma'')$ must wind around the entire length of $T_1$. By the cone inequality we can fill $\gamma$ by a disc $\Sigma$ of area $\lesssim D$.
\par
Consider what $(\pi \circ F) (\gamma)$ looks like in $B^2_1$. All points in $B^2_1$ with radius in $[1/3, 1/2]$ have winding number $\gtrsim N$ with respect to $(\pi \circ F) (\gamma)$. Thus,

$$area(F(\Sigma)) \ge area(\pi \circ F (\Sigma)) \gtrsim N$$

Where area is counted with multiplicity here. Putting together this observation with $area(F(\Sigma)) \le area(\Sigma)D \lesssim D^2$, we get $D \gtrsim N^{1/2}$.  
  
 \end{proof}

\begin{figure}[h]
\centering
\includegraphics[width=0.8\textwidth]{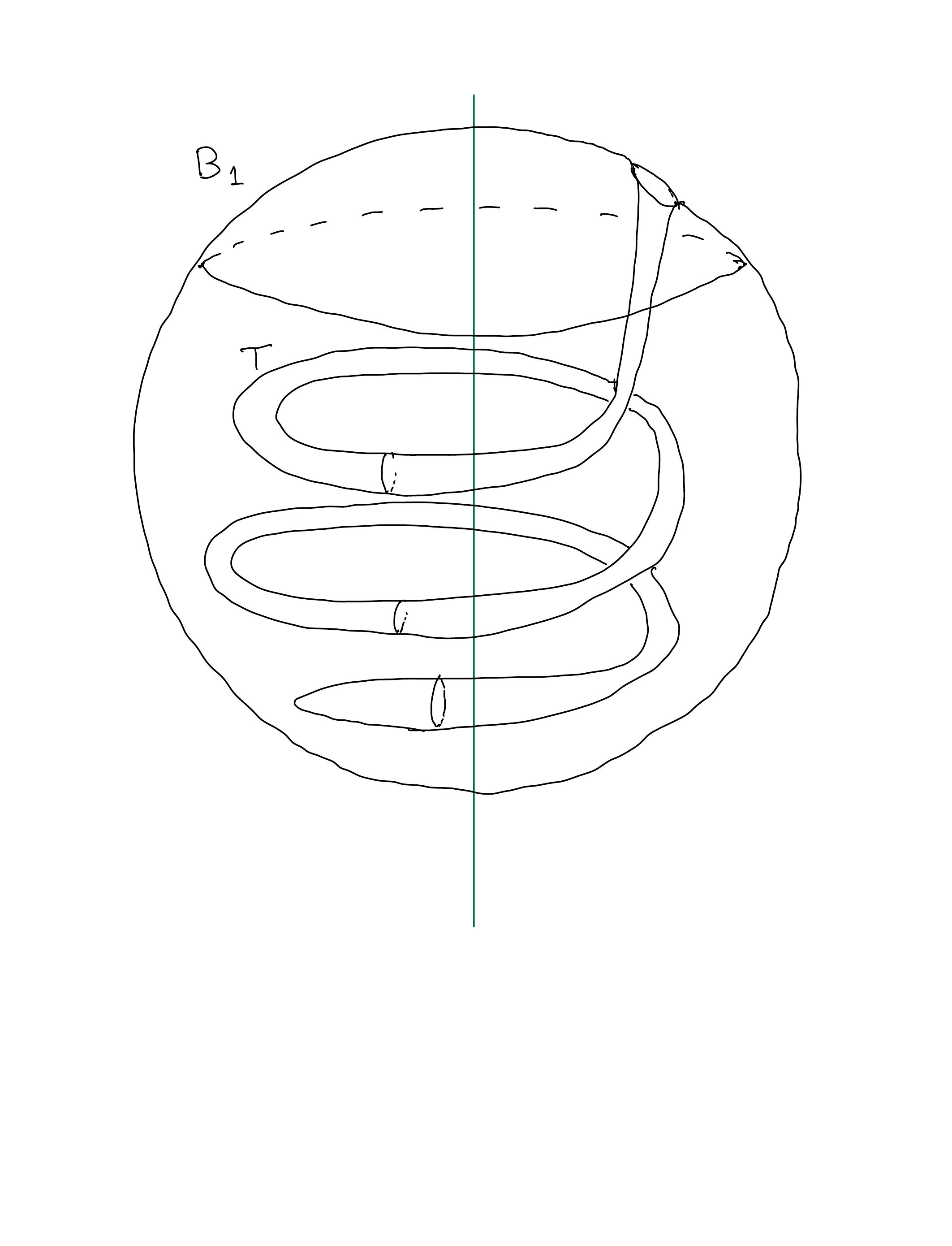}
\caption{Sketch of $U_2$}
\end{figure}

\end{document}